\documentclass[reqno,12pt]{amsart}
\usepackage{txfonts}
\usepackage{mathrsfs}
\usepackage{amsmath}
\usepackage{amssymb}

\usepackage{amsthm}
\usepackage{verbatim}
\usepackage{latexsym, bm}

\title[Normal bundles]{Normal bundles on the exceptional sets of simple small resolutions}

\author[Rong Du]{Rong Du$^{\dag}$}
\address{School of Mathematical Sciences\\
Shanghai Key Laboratory of PMMP\\
East China Normal University\\
Rm. 312, Math. Bldg, No. 500, Dongchuan Road\\
Shanghai, 200241, P. R. China} \email{rdu@math.ecnu.edu.cn}

\author[Xinyi Fang]{Xinyi Fang}

\address{School of Mathematical Sciences\\
Shanghai Key Laboratory of PMMP\\
East China Normal University\\
No. 500, Dongchuan Road\\
Shanghai, 200241, P. R. China}

\email{2315885681@qq.com}

\thanks{$^{\dag}$ The Research Sponsored by the National Natural Science Foundation of China (Grant No. 11471116, 11531007) and Science and Technology Commission of Shanghai Municipality (Grant No. 18dz2271000).}

\theoremstyle{definition}
\newtheorem{theorem}[subsection]{Theorem}
\newtheorem{lemma}[subsection]{Lemma}
\newtheorem{definition}[subsection]{Definition}
\newtheorem{proposition}[subsection]{Proposition}
\newtheorem{example}{Example}[section]
\newtheorem{corollary}[subsection]{Corollary}
\newtheorem{remark}[subsection]{Remark}

\newfont{\drnew}{wncyr10}

\allowdisplaybreaks

\def\dashfill{\leaders\hbox{\hbox to 3.25pt{\hrulefill}\hspace*{2pt}\hbox to 3.25pt{\hrulefill}}\hfill}
\newcommand{\CITE}[1]{{[#1]}}
\let\cite=\CITE

\begin{document}

\begin{abstract}
We study the normal bundles of the exceptional sets of isolated simple small singularities in the higher dimension when the Picard group of the exceptional set is $\mathbb{Z}$ and the normal bundle of it has some good filtration. In particular, for the exceptional set is a projective space with the split normal bundle, we generalized Nakayama and Ando's results to higher dimension. Moreover, we also generalize Laufer's results of rationality and embedding dimension to higher dimension.
\end{abstract}

\maketitle

\vspace{1cm}
\section{\textbf{Introduction}}

Let $X$ be a complex manifold of dimension $n$ containing a compact smooth irreducible analytic submanifold $Z$ of dimension $p$. The submanifold $Z$ is called exceptional or contractible, if there exists a birational proper morphism $\varphi: X\rightarrow Y $ whose exceptional set is $Z$, where $Y$ may be an algebraic space or an analytic space. If the subset $Z$ of $X$ is of codimension greater than or equal to $2$,  the morphism $\varphi$ is called a small contraction, and the pair $(\varphi(X), \varphi(Z))$ is called the small singularity. Moreover, if dim $\varphi(Z)=0$, then the singularity is called isolated simple small singularity.

When $X$ is a smooth projective surface, an algebraic curves $C$ is exceptional if and only if its normal bundle $N_{C/X}$ is negative. When dim$X \geq 3$, Grauert has shown that if the normal bundle of algebraic set $Z$, $N_{Z/X}$, is negative, then $Z$ is exceptional. But the reverse is not true. For example, we know many exceptional curves whose normal bundles are not negative (cf. \cite{La1}). However, Nakayama and Ando have some very interesting works on the normal bundles of exceptional curves in the higher dimension (cf. \cite{An1},\cite{An2}, \cite{Na}), especially for the exceptional set is a rational curve.

\begin{theorem}(cf. \cite{An2})\label{P1}
Let $X$ be a nonsingular projective variety of dimension $N\ge 3$ over $\mathbb{C}$, and let $\mathbb{P}^1\in X$. Assume that surjective morphism $f: X\rightarrow Y$ is a contraction map with the $\mathbb{P}^1$ as the exceptional set. Let the normal bundle
\[N_{\mathbb{P}^1/X}=\oplus_{i=1}^{N-1}\mathscr{O}_{\mathbb{P}^1}(-a_i), ~(a_1\le\cdots\le a_{N-1}).\]
Then we have the following inequality holds.
\[2(a_1+\cdots+a_h)+(a_{h+1}+\cdots+a_{N-1})\ge 0,\]
for any $1\le h\le N-1$.
\end{theorem}

From another perspective, when $X$ is a smooth projective variety such that the canonical bundle $K_X$ is not numerically effective. According to the Mori's theory, there exists a contraction morphism $\varphi: X\rightarrow Y$ such that the fibers of $\varphi$ are connected. The major goal of Mori's theory is to construct a minimal model for each nonuniruled birational equivalence class of varieties. Studying the structure of small contraction maps, as well as finding the associated ``surgery operation" (flip), is of great importance for the minimal model program. To this end, it is essential to analyze the exceptional loci and their normal sheaves (\cite{Mo}, \cite{Ka}). When $n=4$,
Kawamata (\cite{Ka}) proved that the exceptional locus $E$ of $\varphi$ is a disjoint union of its irreducible components $E_i$ such that $E_i\cong \mathbb{P}^2$. In general, when $n=2k$, Zhang (\cite{Zh}) proved that if each irreducible component $E_i$ of the exceptional locus $E$ of $\varphi$ is a smooth subvariety of dimension $k$, then $E_i\cong \mathbb{P}^k$. He also gave the result when $n=5$ and $p=3$. Later, Su-Zhao studied small contractions of odd dimensional smooth complex projective varieties. More precisely,  they showed that when $n=2k-1$, if each irreducible component $E_i$ of the exceptional locus $E$ of $\varphi$ is a smooth subvariety of dimension $k$, then $E_i$ is isomorphic to $\mathbb{P}^k$, quadratic hypersurface $Q^k\subseteq\mathbb{P}^{k+1}$, or a linear $\mathbb{P}^{k-1}$-bundle over a smooth curve. Moreover, if dim $\varphi(E)=0$, the third case never happens. Since dim$E\ge \frac{1}{2}$ dim$X$ by \cite{Wi}, we know that the structure of $E$ is very simple when the dimension of $E$ is the minimum for a small contraction $\varphi: X\rightarrow Y$. Therefore studying the projective space as the exceptional set is a basic and important research direction, especially for the normal bundle of the exceptional set. Kachi (\cite{Kac}) study the normal bundle of $\mathbb{P}^2$ as the exceptional set of some special flip contraction. However, as far as we know, such kind of result is very few for higher dimensional exceptional set . The simple reason is that the vector bundles on $\mathbb{P}^k$ are not known very clearly. So, even if the normal bundle of the exceptional set splits, the explicit split type is still unknown.

The purpose of this paper is to study the normal bundle of the exceptional set of isolated simple small singularities in the higher dimension case when the Picard group of the exceptional set is $\mathbb{Z}$ and the normal bundle of it has some good filtration. The typical example is that the exceptional set is projective space with splitting normal bundle. So we generalized Nakayama (\cite{Na}) and Ando's (\cite{An2}) results to higher dimension.

%
	
\begin{definition}
Let $E$ be a holomorphic vector bundle on a variety $Z$ of rank $r$ and Pic$Z\cong\mathbb{Z}$. If there exists a filtration $$E=\mathscr{F}_0\supset \mathscr{F}_1 \supset \cdots\supset \mathscr{F}_{r}=0$$ with all $\mathscr{F}_{i-1}/\mathscr{F}_{i}$ $(1\le i\le r)$ are invertible sheaves and $\mathscr{F}_{i-1}/\mathscr{F}_{i}=\mathscr{O}_Z(a_i)$, where $a_i\in \mathbb{Z}$, we call $E$ has a good filtration.
\end{definition}

Let $X$ be a complex manifold with a compact smooth irreducible analytic submanifold $Z$ with Pic $Z$ $\cong \mathbb{Z}$. Suppose dim$X=n$, dim$Z=p$, and $n-p\ge 2$.

Our main result is as follows.

\begin{theorem}\label{mainth}
  Let $X$ be a complex manifold of dimension $n$ and $Z$ be an exceptional set of dimension $p$ $(n-p\ge 2)$ of an isolated simple small singularity such that Pic $Z\cong \mathbb{Z}$. Let $I_Z$ be the ideal sheaf of $Z\subseteq X$. If the conormal bundle $I_Z/I_Z^2$ has a good filtration, then

\[\sum\limits_{(t_1,t_2,\cdots,t_{n-p})\in T \atop}
	a_1^{t_1} a_2^{t_2} \cdots a_{n_p}^{t_{n-p}}
	\geq 0
	\]
	and
	\[\sum\limits_{(t_1,t_2,\cdots,t_{n-p})\in T \atop}
	a_1'^{t_1} a_2'^{t_2} \cdots a_{n-p}'^{t_{n-p}}
	\geq 0,
	\]
	where
	\[
	T=\{(t_1,t_2,\cdots,t_{n-p})\in \mathbb{Z}^{n-p} \bigm| t_1\geq 0, t_2\geq 0, \cdots t_{n-p}\geq 0,~
	t_1+t_2+\cdots+t_{n-p}=p\},\]
	and \[
	a_{m}'=
	\left\{
	\begin{array}{cll}
	&2a_{m}, & {1\le m\le h;}\\
	&a_{m}, &  \text{otherwise,}
	\end{array}
	\right.\]
	for any $1 \le h \le {n-p-1}$.
\end{theorem}
	\vspace{.5cm}
\begin{remark}  There are many new restrictions of those $a_i$'s indeed if the dimension of $Z$ is odd. For example, fix the same assumption as Theorem \ref{mainth} and let dim$Z=3$, then \\
	\[
	\sum\limits_{i=1}^{n-3} a_i^3+\sum\limits_{1\le i,j\le n-3\atop i\neq j}
	a_i^2a_j + \sum\limits_{1\le i,j,k\le n-3\atop i\neq j\neq k}a_ia_ja_k \geq 0
	\]
	and
	\[
	\sum\limits_{i=1}^{n-3} a_i'^3+\sum\limits_{1\le i,j\le n-3\atop i\neq j}
	a_i'^2a_j' + \sum\limits_{1\le i,j,k\le n-3\atop i\neq j\neq k}a_i'a_j'a_k' \geq 0,
	\]
	where
    \[
	a_m'=
	\left\{
	\begin{array}{cll}
	&2a_m, & {1\le m\le h}\\
	&a_m, & {\text{otherwise}}
		\end{array}
	\right.
	\]
	for any  $1\le h \le {n-4}$.\\
\end{remark}

\begin{corollary}
Let $X$ be a complex manifold of dimension $n$ and $\mathbb{P}^p$ be the exceptional set of an isolated simple small singularity, $(n-p\ge 2)$. If the normal bundle $N_{\mathbb{P}^p/X}\cong \oplus_{i=1}^{n-p}\mathscr{O}_{\mathbb{P}^p}(-a_i)$, then we have a system of inequalities
\[\sum\limits_{(t_1,t_2,\cdots,t_{n-p})\in T \atop}
	a_1'^{t_1} a_2'^{t_2} \cdots a_{n-p}'^{t_{n-p}}
	\geq 0,	\]
	where
	\[
	T=\{(t_1,t_2,\cdots,t_{n-p})\in \mathbb{Z}^{n-p} \bigm| t_1\geq 0, t_2\geq 0, \cdots t_{n-p}\geq 0,~
	t_1+t_2+\cdots+t_{n-p}=p\},\]
	and $a_i'=a_i$ or $2a_i$.
\end{corollary}	

If we let $p=1$, then we can get Nakayama and Ando's result, Theorem \ref{P1}.

\begin{corollary}\label{cod2}
Fix the same assumption as Theorem \ref{mainth} and suppose dimension of $Z$ is odd of codimension $2$ in $X$, then $a_1+a_2\geq 0$ and  $2a_1+a_2\geq 0$.
\end{corollary}

In \cite{La1}, Laufer gave a sufficient condition for the rationality of an isolated singularity when the exceptional set is $\mathbb{P}^1$. Moreover, he also calculated the Hilbert function of the singularity and gave the embedding dimension of it especially.  We generalize Laufer's results to higher dimension as follows.
\begin{theorem}
	Let $Z\cong \mathbb{P}^p$ be an exceptional set in the n-dimensional manifold $X$. Suppose  the normal bundle $N_{\mathbb{P}^p/X}\cong \oplus_{i=1}^{n-p}\mathscr{O}_{\mathbb{P}^p}(-a_i)$ and $a_i\geq 0$ for $1 \le i \le n-p$. Let $\varphi: (X, Z)\rightarrow (Y, y)$ is the contraction morphism, then  $(Y,y)$ is a rational singularity. Let $m_y$ be the maximal ideal of $Y$ at $y$. Let $h(r)=dim(m_y^r/m_y^{r+1})$ be the Hibert function for $Y$ at $y$, then
	\[
	h(r)=\sum\limits_{i_1+i_2+\cdots+i_{n-p}=r \atop i_1\geq 0,i_2\geq 0,\cdots,i_{n-p}\geq 0}{p+i_1a_1+\cdots+i_{n-p}a_{n-p} \choose p}.
	\]
	In particular, at $y$ the embedding dimension of $Y$ is $\sum\limits_{i=1}^{n-p}{p+a_i \choose a_i}$.
	\end{theorem}

\section{\textbf{Normal bundles of the exceptional sets}}
	
	Let $X$ be a complex manifold of dimension $n$ and $Z$ be an exceptional set of $X$ of dimension $p$ with Pic $Z\cong \mathbb{Z}$. Let $I_Z$ be the ideal sheaf of $Z\subseteq X$ and $I_Z/I_Z^2$ be the conormal bundle of $Z$ in $X$.
	
\begin{definition}
For two ideals  $\mathscr{F}$   and  $\mathscr{J}$  with  ${I_Z} \supset \mathscr{J}\supset \mathscr{F}$  and Supp $\mathscr{O}_X/\mathscr{F}=Z$, there exists a filtration $\mathscr{J}=\mathscr{F}_0 \supset \mathscr{F}_1 \supset \cdots\supset \mathscr{F}_{r} \supset \mathscr{F}$ with every $\mathscr{F}_{i-1}/\mathscr{F}_{i}$ is a locally free $\mathscr{O}_Z-module$  and $\mathscr{F}_{r}/\mathscr{F}$ is a zero-dimensional sheaf. Then we define  $$\text{length}\mathscr{J}/\mathscr{F}=\sum_{i=1}^{r} {\text{rank} (\mathscr{F}_{i-1}/\mathscr{F}_{i})}$$ and $$ c_j(\mathscr{J}/\mathscr{F})={c_j(\bigoplus_{i=1}^{r} {\mathscr{F}_{i-1}/\mathscr{F}_{i}})}, ~(1 \le j\le p).$$ Clearly, these are independent of the choice of a filtration by the properties of Chern classes.
\end{definition}
	
	\begin{definition}
Two coherent sheaves $\mathscr{F}$  and $\mathscr{G}$ over a smooth manifold $Z$ of dimension $p$ is called to be numerical equivalent (denoted $\mathscr{F \equiv  G})$, if length $\mathscr{F}=$ length $\mathscr{G}$ and $c_i(\mathscr{F})=c_i(\mathscr{G})$ for all $i$ $(1 \le i \le p)$.
	\end{definition}

	If the conormal bundle $I_Z/I_Z^2$ has a good filtration, then we have corresponding filtration ${I_Z}=\mathscr{F}_0 \supset \mathscr{F}_1 \supset \cdots\supset \mathscr{F}_{n-p}={I_Z}^2$ with all $\mathscr{F}_{i-1}/\mathscr{F}_{i}$ are invertible sheaves and $\mathscr{F}_{i-1}/\mathscr{F}_{i}\cong \mathscr{O}_Z(a_i)$.

For any  $1\le h \le {n-p-1}$, let $J =\mathscr{F}_{h}$. Put
\begin{eqnarray*}
B(r)=\{(m_1,m_2,\cdots,m_{n-p}) \in \mathbb{Z}^{n-p} \bigm| m_i\geq 0\  (1 \le i \le n-p),\\
(m_1+m_2+\cdots+m_h)+2(m_{h+1}+\cdots+m_{n-p})=r\}.
\end{eqnarray*}

By the similar argument by Ando in \cite{An1}, Proposition 2.5,	for any nonnegative integer integer $r$, we have the following result.  We mimic Ando's proof as follows in order to keep this paper self-contained.	
	
	\begin{proposition}\label{grade}
	\[
	J^r/I_ZJ^r \equiv \bigoplus_{(m_1,m_2,\cdots,m_{n-p})\in B(2r) \atop}
	\mathscr{O}_Z(m_1a_1+ \cdots + m_{n-p}a_{n-p})
	\]
	\[
	I_ZJ^r/J^{r+1} \equiv \bigoplus_{(m_1,m_2,\cdots,m_{n-p})\in B(2r+1) \atop}
	\mathscr{O}_Z(m_1a_1+ \cdots + m_{n-p}a_{n-p})
	\]
	\[
	I_Z^r/I_Z^{r+1} \equiv \bigoplus_{m_1+\cdots+m_{n-p}=r \atop}
	\mathscr{O}_Z(m_1a_1+ \cdots + m_{n-p}a_{n-p})
	\]
\end{proposition}
\begin{proof}
On a small neighborhood $U$ of $p\in Z$, we can find a set of functions $g_1, \cdots , g_{n-p}$ on $U$ corresponding to the above filtration with $g_{i}\in \Gamma(U,\mathscr{F}_{i-1})$ and the class of $g_{i}$ modulo $\mathscr{F}_{i}$ is a local base of the invertible sheaf $\mathscr{F}_{i-1}/\mathscr{F}_{i}$ $(1 \le i \le n-p)$. Clearly
\begin{eqnarray*}
I_Z \bigm|U=(g_1, \cdots , g_{n-p}),
\end{eqnarray*}

\begin{eqnarray*}
J \bigm|U=(g_1, \cdots , g_{n-p})^2+(g_{h+1}, \cdots , g_{n-p}).
\end{eqnarray*}

Next we only prove the second assertion. Put $G_{k}=I_Z^{2k-2}J^{r-k+2}+I_Z^{2k+1}J^{r-k}$, and $H_{k}=I_Z^{2k-1}J^{r-k+1}\cap J^{r+1}$. We formally put $J^0=I_Z^0=I_Z^{-1}=I_Z^{-2}=\mathscr{O}_X$. First we prove $H_{k} \supset G_{k}$. It is enough to show this on $U$. Let $x^{\alpha}$ be a monomial in $g_1, \cdots , g_{h}$ of degree $\alpha$, $y^{\beta}$ be a monomial in $g_{h+1}, \cdots , g_{n-p}$ of degree $\beta$. Since $G_{k}$ and $H_{k}$ are generated by the monomials in the form $x^{\alpha}y^{\beta}$, it is enough to show that if $x^{\alpha}y^{\beta}\in H_{k}$, then $x^{\alpha}y^{\beta}\in G_{k}$. Note that $x^{\alpha}y^{\beta}\in I_Z^{i}J^{j}$ if and only if $\alpha+\beta\geq i+j$ and $\alpha+2\beta\geq i+2j$. Assume $x^{\alpha}y^{\beta}\in H_{k}$ but $x^{\alpha}y^{\beta}\notin G_{k}$, then $x^{\alpha}y^{\beta}\in I_Z^{2k-1}J^{r-k+1}$, but $x^{\alpha}y^{\beta}\notin I_Z^{2k+1}J^{r-k}$, so we have $\alpha+\beta=r+k$ and $\alpha+2\beta\geq 2r+1$. Since $x^{\alpha}y^{\beta}\notin I_Z^{2k-2}J^{r-k+2}$, we have $\alpha+2\beta=2r+1$. But $x^{\alpha}y^{\beta} \in J^{r+1}$ leads to $\alpha+2\beta \geq 2r+2$, which is impossible. Thus we have $H_{k} \supset G_{k}$.

Since $I_Z^{2k-2}J^{r-k+2} \subset H_{k}$, we have $H_{k}+I_Z^{2k+1}J^{r-k}=G_{k}$. Put $F_{k}=G_{k}/H_{k}$. By the isomorphism theorem , $F_{k} \cong I_Z^{2k+1}J^{r-k}/H_{k+1}$. Thus we have
\begin{eqnarray*}
F_{k}/F_{k+1}\cong I_Z^{2k+1}J^{r-k}/G_{k+1},~ (0 \le k \le r).
\end{eqnarray*}

On the other hand
\begin{eqnarray*}
S^{2k+1}(I_Z/J)\otimes S^{r-k}(J/I_Z^2)\cong I_Z^{2k+1}J^{r-k}/G_{k+1},~(0 \le k \le r).
\end{eqnarray*}

Thus we have
\begin{eqnarray*}
F_{k}/F_{k+1}\cong S^{2k+1}(I_Z/J)\otimes S^{r-k}(J/I_Z^2), ~(0 \le k \le r).
\end{eqnarray*}

By the definition of numerical equivalence, we have
\begin{eqnarray*}
F_0/F_{r+1} \equiv \bigoplus_{t=0}^{k} S^{2k+1}(I_Z/J)\otimes S^{r-k}(J/I_Z^2).
\end{eqnarray*}

Because $F_0 \cong I_ZJ^{r}/J^{r+1}$, $F_{r+1} \cong I_Z^{2r}J/I_Z^{2r} \cap J_{r+1}=0$. We have
\begin{eqnarray*}
I_ZJ^r/J^{r+1} &\equiv& \bigoplus_{t=0}^{k} S^{2k+1}(I_Z/J)\otimes S^{r-k}(J/I_Z^2)\\
&\equiv& \bigoplus_{(m_1,m_2,\cdots,m_{n-p})\in B(2r+1) \atop}
\mathscr{O}_Z(m_1a_1+ \cdots + m_{n-p}a_{n-p}).
\end{eqnarray*}

For the first assertion, put $G_{k}=I_Z^{2k-3}J^{r-k+2}+I_Z^{2k}J^{r-k}$, and  $H_{k}=I_Z^{2k-2}J^{r-k+1}\cap I_ZJ^{r}$. By repeating the above process, we are done. For the last assertion , put $J=I_Z$ and we are done.
\end{proof}

	The following combinatoric results will be used in the proof of the main theorem.
	\begin{theorem}  (Euler's Finite Difference Theorem, cf. \cite{Go-Qu}, Section 10, Page 45)\label{EFDT}
	\[
	\sum\limits_{t=0}^{k}(-1)^t {k \choose t} t^j=\left\{
	\begin{array}{cll}
	&0, &{ 0\le j<k;}\\
	&(-1)^k k!,  & {j=k}.
		\end{array}
	\right.
	\]
\end{theorem}
	
\begin{theorem} (Variations of Theorem \ref{EFDT}, cf. \cite{Go-Qu}, Section 10, Page 45)\label{VT}
	\[
	\sum\limits_{t=0}^{k}(-1)^t {k \choose t} (k+i-t)^k =k!  ~ (i \geq 0).
	\]
\end{theorem}

	\begin{lemma}\label{Comb}
For any positive integer $i,j,k$,
	
	\[\sum\limits_{m_1+\cdots+m_i=j \atop}m_1^k={i+j+k-2 \choose i+k-1}+u_1{i+j+k-3 \choose i+k-1}+\cdots+u_{k-1}{i+j-1 \choose i+k-1},\]
	where
	\[1+u_1+\cdots+u_{k-1}=k!.\]
	\end{lemma}

\begin{proof}	
\begin{eqnarray*}
\sum\limits_{m_1+\cdots+m_i=j \atop}m_1^k&=&1^k{i+j-3 \choose i-2}+\cdots+k^k{i+j-k-2 \choose i-2}+\cdots+j^k{i-2 \choose i-2}\\
	&=&{i+j-2 \choose i-1}+\cdots+\big(k^k-(k-1)^k \big) {{i+j-k-1} \choose {i-1}} +\cdots\\
      &&+\big(j^k-(j-1)^k \big){i-1 \choose i-1}\\	
	&=&{i+j-1 \choose i} + \cdots+\big(k^k-2(k-1)^k+(k-2)^k\big)  {{i+j+k} \choose i}\\
    &&+\cdots+\big(j^k-2(j-1)^k+(j-2)^k \big){i \choose i}\\
	&=&\cdots\cdots\\	
	&=&{i+j+k-3 \choose i+k-2}+\cdots+\sum\limits_{t=0}^{k}(-1)^t {k \choose t} (k-t)^k {i+j-2 \choose i+k-2}+\\
&&\sum\limits_{t=0}^{k}(-1)^t {k \choose t} (k+1-t)^k {i+j-3 \choose i+k-2}+\cdots\\
&&+\sum\limits_{t=0}^{k}(-1)^t {k \choose t} (k+j-k-t)^k{i+k+2 \choose i+k+2}~ (\text{Theorem} \ref{VT})\\	
	&=&{i+j+k-3 \choose i+k-2}+\cdots+k!{i+j-2 \choose i+k-2}+k!{i+j-3 \choose i+k-2}\\
&&+\cdots+k!{i+k-2 \choose i+k-2}\\
	&=&{i+j+k-2 \choose i+k-1} +u_1{i+j+k-3 \choose i+k-1}+\cdots+u_{k-1}{i+j-1 \choose i+k-1},
\end{eqnarray*}
	where $1+u_1+\cdots+u_{k-1}=k!.$
\end{proof}
	$\textbf{Proof of Theorem \ref{mainth}:}$			
	\begin{proof}By Proposition \ref{grade}, we have
\begin{eqnarray*}
	ch(\mathscr{O}_X/I_Z^r)&=&{n+r-p-1 \choose n-p}+\sum\limits_{i=0}^{r-1}\sum\limits_{m_1+\cdots+m_{n-p}=i \atop}(m_1a_1+\cdots+m_{n-p}a_{n-p})+\\
	&&\frac{1}{2!}\sum\limits_{i=0}^{r-1}\sum\limits_{m_1+\cdots+m_{n-p}=i \atop}(m_1a_1+\cdots+m_{n-p}a_{n-p})^2+\cdots\\ &&+\frac{1}{p!}\sum\limits_{i=0}^{r-1}\sum\limits_{m_1+\cdots+m_{n-p}=i \atop}(m_1a_1+\cdots+m_{n-p}a_{n-p})^p.
\end{eqnarray*}	
where $	ch(\mathscr{O}_X/I_Z^r)$ is the exponential Chern character of $	\mathscr{O}_X/I_Z^r$.
By Lemma \ref{Comb},
we know that for any$1 \le j \le p$,
\begin{eqnarray*}
	\sum\limits_{i=0}^{r-1}\sum\limits_{m_1+\cdots+m_{n-p}=i \atop}(m_1a_1+\cdots+m_{n-p}a_{n-p})^j
	\end{eqnarray*}
	is a polynomial in a single indeterminate $r$ of degree $n-p+j$. Therefore 		
\begin{eqnarray*}
&\frac{1}{p!}\sum\limits_{i=0}^{r-1} \sum\limits_{m_1+\cdots+m_{n-p}=i \atop}(m_1a_1+\cdots+m_{n-p}a_{n-p})^p\\
	=&\frac{1}{p!}\sum\limits_{(t_1,t_2,\cdots,t_{n-p})\in T \atop}{p \choose t_1}\cdots {p-t_1-\cdots-t_{n-p-1} \choose t_{n-p}}\\
	&\cdot a_1^{t_1} a_2^{t_2} \cdots a_{n-p}^{t_{n-p}}
	\sum\limits_{i=0}^{r-1} \sum\limits_{m_1+\cdots+m_{n-p}=i \atop}(m_1^{t_1}\cdots m_{n-p}^{t_{n-p}})^p,
\end{eqnarray*}
	where
	\[
	T=\{(t_1,t_2,\cdots,t_{n-p})\in \mathbb{Z}^{n-p} \bigm| t_1\geq 0, t_2\geq 0, \cdots t_{n-p}\geq 0,~
	t_1+t_2+\cdots+t_{n-p}=p\}.\]
	is a polynomial in a single indeterminate $r$ of degree n.
By Lemma \ref{Comb}, for any $(t_1,t_2,\cdots,t_{n-p})\in T$, we know the leading coefficient of
	\[\sum\limits_{i=0}^{r-1} \sum\limits_{m_1+\cdots+m_{n-p}=i \atop}(m_1^{t_1}\cdots m_{n-p}^{t_{n-p}})\in \mathbb{C}[r]\]
	
	is  $t_1!t_2! \cdots t_{n-p}!$.

	Thus the leading coefficient of the above polynomial is\\
		\[
		\sum_{(t_1,t_2,\cdots,t_{n-p})\in T \atop}
		a_1^{t_1} a_2^{t_2} \cdots a_{n-p}^{t_{n-p}}.\]
		up to a positive constant.\\
	Therefore by generalized Grothendieck-Hirzebruch-Riemann-Roch theorem, on the one hand, we can express the Euler characteristic $\chi(\mathscr{O}_X/I_Z^r)$ as a polynomial in a single indeterminate $r$ , the leading coefficient of it is
	\[
	\sum_{(t_1,t_2,\cdots,t_{n-p})\in T \atop}
	a_1^{t_1} a_2^{t_2} \cdots a_{n-p}^{t_{n-p}}.\]
	up to a positive constant.

	On the other hand,\[
	\chi(\mathscr{O}_X/I_Z^r) \geq \sum_{i=1}^p (-1)^idimH^i(\mathscr{O}_X/I_Z^r).
	\]
	However, because Z is the exceptional set,
 $$\lim_{\leftarrow}H^i(\mathscr{O}_X/I_Z^r)\cong (R^i\varphi_*\mathscr{O}_X^\wedge)_y,$$ by the holomporphic functions theorem (cf. \cite{Knu}, Chapter 5, Theorem 3.1), where $\varphi: (X, Z)\rightarrow (Y, y)$ is the contraction morphism. Therefore, $$\text{dim}\lim_{\leftarrow}H^i(\mathscr{O}_X/I_Z^r)\le +\infty.$$ Thus there exists some constant $M$ which is independent of $r$ such that dim$H^i(\mathscr{O}_X/I_Z^r) \le M$ for all $i$ $(1\le i \le p)$.\\
	Thus
	\[
	\sum_{(t_1,t_2,\cdots,t_{n-p})\in T \atop}
	a_1^{t_1} a_2^{t_2} \cdots a_{n-p}^{t_{n-p}}
	\geq 0.
	\]			
	Moreover, for any $1 \le h \le {n-p-1}$, put $J =\mathscr{F}_{h}$, we can consider the sheaf $\mathscr{O}_X/J^r$, then by the similar argument as above, we only need to focus on the leading coefficient of the polynomial $\chi(\mathscr{O}_X/J^r)$. Similarly, we have

\[\sum\limits_{(t_1,t_2,\cdots,t_{n-p})\in T \atop}
	a_1'^{t_1} a_2'^{t_2} \cdots a_{n-p}'^{t_{n-p}}
	\geq 0,
	\]			
	where\[
	a_{m}'=
	\left\{
	\begin{array}{cll}
	&2a_{m}, &{1\le m\le h;}\\
	&a_{m},  &   {\text{otherwise}.}
	\end{array}
	\right.\]
\end{proof}

If the exceptional set is $\mathbb{P}^p$ and the normal $N_{\mathbb{P}^p/X}$ bundle splits, then $I_Z/I_Z^2$ can have several good filtrations. So we have many inequalities.

	\begin{corollary}
Let $X$ be a complex manifold of dimension $n$ and $\mathbb{P}^p$ be the exceptional set of an isolated simple small singularity, $(n-p\ge 2)$. If the normal bundle $N_{\mathbb{P}^p/X}\cong \oplus_{i=1}^{n-p}\mathscr{O}_{\mathbb{P}^p}(-a_i)$, then we have a system of inequalities
\[\sum\limits_{(t_1,t_2,\cdots,t_{n-p})\in T \atop}
	a_1'^{t_1} a_2'^{t_2} \cdots a_{n-p}'^{t_{n-p}}
	\geq 0,
	\]
	where
	\[
	T=\{(t_1,t_2,\cdots,t_{n-p})\in \mathbb{Z}^{n-p} \bigm| t_1\geq 0, t_2\geq 0, \cdots t_{n-p}\geq 0,~
	t_1+t_2+\cdots+t_{n-p}=p\},\]
	and $a_i'=a_i$ or $2a_i$.
\end{corollary}

\begin{remark}
Set $p=1$, we can get Nakayama and Ando's result, Theorem \ref{P1}.
\end{remark}

	$\textbf{Proof of Corollary \ref{cod2}:}$
\begin{proof}
	By Theorem \ref{mainth}, we know that
	\[\sum\limits_{i=0}^{2n-1}a_1^ia_2^{2n-1-i}\geq 0\] and
	\[\sum\limits_{i=0}^{2n-1}(2a_1)^ia_2^{2n-1-i}\geq 0,\]
where dim$X=2n+1$.

	Since \[
	\sum_{i=0}^{2n-1}a_1^ia_2^{2n-1-i}=(a_1+a_2)(\sum_{i=0}^{2n-1}a_1^{2i} a_2^{2n-2-2i})
	\]	
	and \[\sum\limits_{i=0}^{2n-1}a_1^{2i} a_2^{2n-2-2i}\geq 0,\] we have
	\[a_1+a_2\geq 0.\]

	Similarly, \[2a_1+a_2\geq 0.\]\\
\end{proof}

\section{\textbf{Isolated small singularities}}
 Fixed the same notations as the formal sections. When dim$X\ge 3$, Grauert has shown that if the normal bundle of algebraic set $Z$, $N_{Z/X}$, is negative, then $Z$ is exceptional. But the reverse is not true. Even the first Chern class of the normal bundle is not negative. However, if we consider locally a Stein neighborhood of an isolated Gorenstein singularity such that the canonical divisor of $Z$ is negative, then $c_1(N_{Z/X})<0$. In fact, since a small resolution is crepant and the canonical divisor $\omega_X$ is trivial, by adjunction formula, we have
\[\omega_Z\cong \omega_X\otimes \wedge^r(N_{Z/X}).\]
So $c_1(N_{Z/X})=K_Z$ which is negative.

If the exceptional set of a simple small singularity is $\mathbb{P}^1$, Laufer studied the embedding dimension of such singularity (\cite{La1}). If the exceptional set is  $\mathbb{P}^p$, we can have similar result.
\begin{theorem}
	Let $Z\cong \mathbb{P}^p$ be an exceptional set in the n-dimensional manifold $X$. Suppose  the normal bundle $N_{\mathbb{P}^p/X}\cong \oplus_{i=1}^{n-p}\mathscr{O}_{\mathbb{P}^p}(-a_i)$ and $a_i\geq 0$ for $1 \le i \le n-p$. Let $\varphi: (X, Z)\rightarrow (Y, y)$ is the contraction morphism, then $(Y,y)$ is a rational singularity. Let $m_y$ be the maximal ideal of $Y$ at $y$. Let $h(r)=dim(m_y^r/m_y^{r+1})$ be the Hibert function for $Y$ at $y$, then
	\[
	h(r)=\sum\limits_{i_1+i_2+\cdots+i_{n-p}=r \atop i_1\geq 0,i_2\geq 0,\cdots,i_{n-p}\geq 0}{p+i_1a_1+\cdots+i_{n-p}a_{n-p} \choose p}.
	\]
	In particular,at $y$ the embedding dimension of $Y$ is $\sum\limits_{i=1}^{n-p}{p+a_i \choose a_i}$.
	\end{theorem}
	
\begin{proof}:
	$I_Z$ is the defining idea of $Z$ in $\mathscr{O}_X$, we formlly put $I_Z=\mathscr{O}_X$. Consider the exact sheaf sequence
	\[
	0\rightarrow I_Z^{r+1}\rightarrow I_Z^r\rightarrow I_Z^r/I_Z^{r+1}\rightarrow 0.
	\]
	Since $a_i\geq 0$ for $1 \le i \le n-p$, $H^j(Z,I_Z^r/I_Z^{r+1})=0$ for all $r\geq 0$ and $1\le j\le p$. Hence $H^j(X,I_Z^{r+1})\rightarrow H^j(X,I_Z^r)$ is onto for all $r\geq 0$ and $1\le j\le p$. By \cite{Gra}, Satz 4.2, $H^j(X,I_Z^r)=0$ for all $r\geq 0$ and $1\le j\le p$. In particular,  $H^j(X,\mathscr{O}_X)=0$ for all $1\le j\le p$. Hence $(Y,y)$ is a rational singularity.

	Then, as in \cite{La2}, $m_y^r\cong \Gamma(Z,I_Z^r)$ and $m_y^r/m_y^{r+1}\cong \Gamma(Z,I_Z^r/I_Z^{r+1})$, hence
\begin{eqnarray*}
	\text{dim}(m_y^r/m_y^{r+1})&=&\text{dim} \Gamma(Z,I_Z^r/I_Z^{r+1})=\text{dim} \Gamma(Z,S^r(I_Z/I_Z^2))\\
	                    &=&\sum\limits_{i_1+i_2+\cdots+i_{n-p}=r \atop i_1\geq 0,i_2\geq 0,\cdots,i_{n-p}\geq 0}{p+i_1a_1+\cdots+i_{n-p}a_{n-p} \choose p}
	\end{eqnarray*}
	In particular, at $y$ the embedding dimension of $Y$ is $h(1) =\sum\limits_{i=1}^{n-p}{p+a_i \choose a_i}$.
	\end{proof}

\begin{example}
	Let $U$, $V$ and $W$ be $\mathbb{C}^4$ with coordinates $(t_1,t_2,x_1,x_2)$ , $(s_1,s_2,y_1,y_2)$ and $(w_1,w_2,z_1,z_2)$. We construct $X$ and $Z\cong \mathbb{P}^2$ by the following transition functions
	\[
	\left\{
	\begin{array}{cllllll}
	&x_1=\frac{z_2}{z_1}=\frac{1}{y_1}\\
	
	&x_2=\frac{1}{z_1}=\frac{y_2}{y_1}\\
	
%
	
	&t_1=z_1^5w_1=y_1^5s_1\\
	
	&t_2=z_1w_2=y_1s_2
	\end{array}
	\right.\]
$Z$ is contracted by
\[
\left\{
\begin{array}{cllllll}
&v_1=t_1=z_1^5w_1=y_1^5s_1\\

&v_2=t_2^2=w_2^2z_1^2=s_2^2y_1^2\\

&v_3=x_2^5t_1=w_1=s_1y_2^5\\

&v_4=x_2^2t_2^2=w_2^2=s_2^2y_2^2\\

&v_5=x_1^5t_1=w_1z_2^5=s_1\\

&v_6=x_1^2t_2^2=w_2^2z_2^2=s_2^2
\end{array}
\right.\]
It is easy to check $I_Z/I_Z^2\cong \mathscr{O}_Z(5)\oplus\mathscr{O}_Z(1)$. So $(Y,y)$ is a rational singularity with embedding dimension $24$. Here $\varphi: (X, Z)\rightarrow (Y, y)$ is the contraction morphism of $Z$ (see \cite{An1} Section 3).
	\end{example}
\section*{Acknowledgements}
Both authors would like to thank for the reviewers for pointing out some typos in the original version.

\end{document}